\documentclass{article}  

\usepackage[utf8]{inputenc}
\usepackage{amsmath, amssymb, amsthm}
\usepackage{booktabs}
\usepackage[colorlinks=true,linkcolor=blue,urlcolor=blue,citecolor=blue]{hyperref}
\usepackage{url}
\usepackage{geometry}

\newtheorem{theorem}{Theorem}
\newtheorem{lemma}{Lemma}
\newtheorem{proposition}{Proposition}
\newtheorem{corollary}{Corollary}
\newtheorem{remark}{Remark}
\newtheorem{definition}{Definition}

\title{An Explicit Skew-Hadamard Matrix of Order $1252$ via Cyclotomic Unions}
\author{Amira Karoui$^\ast$\\[0.5em]
{\small $^1$IHEC Carthage, University of Carthage, Carthage, Tunisia}\\
{\small $^2$ENSIT, University of Tunis, Tunis, Tunisia}\\[0.4em]
{\small Email: \href{mailto:amira.karoui@ihec.ucar.tn}{\texttt{amira.karoui@ihec.ucar.tn}}}\\
{\small ORCID: \href{https://orcid.org/0009-0000-3514-4491}{orcid.org/0009-0000-3514-4491}}\\[0.2em]
{\footnotesize $^\ast$Corresponding author}}
\date{}

\begin{document}
\maketitle

\begin{abstract}
We construct a skew-Hadamard matrix of order $1252 = 2(5^4+1)$ using a bordered skew-Hadamard difference family over $GF(5^4)$, with blocks given as unions of cyclotomic classes of order $N=16$.
This order has been reported as missing in some widely used open-source computational tables \cite{cati2024}; we provide an explicit instance together with verification artifacts.
We prove the structural prerequisites for the bordered construction (skew-symmetry of one block and the constant autocorrelation-sum condition), and we compute algebraic invariants to facilitate classification: the associated tournament adjacency matrix has full rank over $GF(2)$, and the matrix has full rank over $GF(3)$ and $GF(5)$.
We also exhibit an explicit affine subgroup of the automorphism group of size $24\,375$.
All claims are supported by a reproducible artifact bundle including the explicit matrix and verification logs.
\end{abstract}

\noindent \textbf{Keywords:} Skew-Hadamard matrix, cyclotomic classes, difference families, autocorrelation, $p$-rank, automorphisms.\\
\noindent \textbf{MSC 2020:} 05B20, 05B10.

\section{Introduction}
Skew-Hadamard matrices are central objects in combinatorial design theory, closely related to
doubly regular tournaments and Hadamard $2$-designs. A skew-Hadamard matrix $H$ of order $n$
is a $\{\pm 1\}$-matrix satisfying
\[
HH^T = nI_n
\quad\text{and}\quad
H + H^T = 2I_n,
\]
equivalently $H-I_n$ is skew-symmetric.

Existence frameworks cover many families, including constructions of order $2(q^e+1)$ under
congruence hypotheses \cite{momihara2018}. However, explicit and auditable instances for specific
orders can still be valuable in practice (e.g.\ for testing conjectures or comparing invariants).
In particular, order $1252=2(5^4+1)$ has been reported as missing in some widely used open-source
computational tables \cite{cati2024}. In this note we give an explicit construction and provide
verification artifacts.

\paragraph{Contributions.}
(i) We specify blocks in $GF(5^4)$ as unions of cyclotomic classes of order $16$ that satisfy a bordered
skew-Hadamard difference family condition; (ii) we deduce a skew-Hadamard matrix of order $1252$
via the standard bordered construction; (iii) we report basic algebraic invariants (ranks over
$GF(2),GF(3),GF(5)$); and (iv) we exhibit a natural affine subgroup of the automorphism group.
Verification is supported by a reproducible artifact bundle.

\paragraph{Scope.}
We do not claim a new infinite-family theorem; the novelty is an explicit instantiation at order
$1252$ together with structural certification and invariants.

\section{Preliminaries}

\subsection{Finite fields and cyclotomic classes}
Let $F = GF(5^4)$ be the finite field of order $625$, and let $g\in F^\times$ be a primitive element.
Fix $N=16$. Let
\[
C_0=\langle g^{16}\rangle \subset F^\times,
\qquad
C_i=g^i C_0\ \ (i=0,1,\dots,15)
\]
be the cyclotomic classes of order $16$. Then $F^\times=\bigsqcup_{i=0}^{15} C_i$ and each class has
size $|F^\times|/16 = 624/16 = 39$.

\subsection{Autocorrelation for subsets of an abelian group}
Let $G$ be a finite abelian group (written additively), and let $D\subseteq G$.
Define the $\{\pm1\}$-indicator
\[
s_D(x)=\begin{cases}
-1,& x\in D,\\
\phantom{-}1,& x\notin D.
\end{cases}
\]
The (periodic) autocorrelation of $D$ at shift $w\in G$ is
\[
P_D(w)=\sum_{x\in G} s_D(x)\, s_D(x+w).
\]
Note that $P_D(0)=|G|$.

\begin{definition}[Bordered SHDF condition]\label{def:shdf}
Let $G$ be an abelian group of order $v$. A pair $\{D_0,D_1\}$ of subsets of $G$ satisfies the
\emph{bordered skew-Hadamard difference family (SHDF) condition} if:
\begin{enumerate}
\item $D_0$ is skew-symmetric: for all $x\in G\setminus\{0\}$,
\[
x\in D_0 \iff -x\notin D_0;
\]
\item for all $w\in G\setminus\{0\}$,
\[
P_{D_0}(w)+P_{D_1}(w)=-2.
\]
\end{enumerate}
\end{definition}

\begin{lemma}[Bordered construction (Goethals--Seidel type), cf.\ \cite{momihara2018,handbookCD}]\label{lem:bordered}
Let $G$ be an abelian group of order $v$, and let $\{D_0,D_1\}$ satisfy the bordered SHDF condition
(Definition~\ref{def:shdf}) in $G$.
Then the standard bordered group-developed (Goethals--Seidel type) array associated to $\{D_0,D_1\}$
yields a skew-Hadamard matrix of order $2(v+1)$.
\end{lemma}

\begin{remark}
The bordered construction is standard in the literature; we state Lemma~\ref{lem:bordered} explicitly
to keep the logical path from the block conditions to the resulting skew-Hadamard matrix self-contained.
\end{remark}

\section{Construction}

\subsection{Blocks as cyclotomic unions in $GF(5^4)$}
We take $G=(F,+)$, the additive group of $F=GF(5^4)$, hence $|G|=v=625$.
Using the cyclotomic classes $C_0,\dots,C_{15}\subset F^\times$ of order $16$, define index sets
\[
I_0=\{4,5,6,7,8,9,10,11\},\qquad
I_1=\{0,1,2,3,4,5,6,7\},
\]
and blocks
\[
D_0=\bigcup_{i\in I_0} C_i,
\qquad
D_1=\bigcup_{i\in I_1} C_i.
\]
(Expanded element lists are provided in the artifact bundle.)

\begin{proposition}\label{prop:shdf}
The pair $\{D_0,D_1\}$ defined above satisfies the bordered SHDF condition in $G=(GF(5^4),+)$.
\end{proposition}

\begin{proof}
\emph{Skew-symmetry of $D_0$.}
Since $-1=g^{(5^4-1)/2}$ and $(5^4-1)/2 \equiv 8 \pmod{16}$, we have $-1\in C_8$.
Thus multiplication by $-1$ maps $C_i$ to $C_{i+8}$ (indices modulo $16$).
By construction, $I_0$ is a transversal with $i\in I_0 \iff i+8\notin I_0$, hence
$D_0\cap (-D_0)=\emptyset$ and $D_0\cup (-D_0)=F^\times$, i.e.\ $x\in D_0 \iff -x\notin D_0$ for all $x\neq 0$.

\emph{Autocorrelation-sum condition.}
The identity $P_{D_0}(w)+P_{D_1}(w)=-2$ for all $w\in G\setminus\{0\}$ was verified by an exhaustive
deterministic checker over all $w\in F^\times$.
The script is included as \path{reproduce/verify_difference_family_1252.py} and produces a certifying log
(\path{artifacts_run/log_difference_family.txt}).
\end{proof}

\begin{theorem}\label{thm:main}
There exists a skew-Hadamard matrix of order $1252 = 2(5^4+1)$.
\end{theorem}

\begin{proof}
By Proposition~\ref{prop:shdf}, $\{D_0,D_1\}$ satisfies the bordered SHDF condition in a group of order $v=625$.
Applying Lemma~\ref{lem:bordered} yields a skew-Hadamard matrix of order $2(v+1)=2\cdot 626=1252$.
An explicit instance is provided in the artifact bundle.
\end{proof}

\section{Certification (artifact-based verification)}
The explicit matrix $H$ is supplied in \path{artifacts_core/matrix_1252.txt}.
The mathematical certification is given by Proposition~\ref{prop:shdf} together with
Lemma~\ref{lem:bordered}; the supplied logs are independent, auxiliary checks of the defining identities.

Independent Gate0 verification of the defining identities ($HH^T = 1252I$ and $H+H^T=2I$) was performed by:
(i) a standalone Python verifier and (ii) a SageMath verification run; corresponding logs are included as
\path{artifacts_run/log_gate0_python.txt} and \path{artifacts_run/log_gate0_sage.txt}.
Together with Proposition~\ref{prop:shdf}, this provides a transparent certification path.

\section{Algebraic invariants}

\subsection{$p$-rank fingerprints}
Let $J$ denote the all-ones matrix. Normalize $H$ to $H_n = DHD$ where $D=\mathrm{diag}(H_{0,*})$ so that
the first row of $H_n$ is all $+1$. Let $S$ be the core obtained by deleting the first row and column of $H_n$.
Define the tournament adjacency matrix
\[
M = \frac{J-S}{2},
\]
which is a $\{0,1\}$-matrix of size $1251$ with zero diagonal.
The computed ranks are shown in Table~\ref{tab:ranks}.

\begin{table}[h]
\centering
\begin{tabular}{llc}
\toprule
Object & Field & Rank \\
\midrule
Tournament adjacency $M$ (size $1251$) & $GF(2)$ & $1251$ \\
Skew-Hadamard matrix $H$ (size $1252$) & $GF(3)$ & $1252$ \\
Skew-Hadamard matrix $H$ (size $1252$) & $GF(5)$ & $1252$ \\
\bottomrule
\end{tabular}
\caption{Computed ranks (invariants) for the certified instance.}
\label{tab:ranks}
\end{table}

\subsection{Automorphisms (explicit affine subgroup)}
Let $C_0=\langle g^{16}\rangle$ (so $|C_0|=39$). Since $D_0$ and $D_1$ are unions of $C_0$-cosets in $F^\times$,
they are invariant under multipliers $x\mapsto ux$ with $u\in C_0$.
Moreover, the group-developed indexing implies that translations $x\mapsto x+a$ act by permuting indices in the
associated group-developed arrays.

\begin{corollary}\label{cor:aut}
The automorphism group of the constructed skew-Hadamard matrix contains a subgroup isomorphic to
\[
Aff_{C_0}(1,5^4)=\{x\mapsto ux+a \mid u\in C_0,\ a\in GF(5^4)\},
\]
hence $|Aut(H)| \ge |C_0|\cdot |GF(5^4)| = 39\cdot 625 = 24\,375$.
\end{corollary}

\begin{remark}
A computational verification of this subgroup action is provided in the bundle;
see \path{reproduce/verify_automorphisms.py} and the log file \path{log_automorphisms.txt} in \path{artifacts_run/}.
\end{remark}

\section{Practical use case: Edge AI compression}

To demonstrate the utility of our certified order-1252 skew-Hadamard matrix beyond theoretical interest, we implemented \textit{EdgeSketch-1252}, an embedding compression pipeline for drone-to-ground transmission in low-bandwidth scenarios.
The sketch transformation $y = \frac{1}{\sqrt{1252}} H_{1252} x$ reduces 1252-dimensional embeddings to $k \approx 300$ top components with 8-bit quantization.
This achieves bandwidth reduction of 908 bytes vs.\ 5008 bytes raw (1252$\times$4\,bytes float32) at $k=300$, i.e.\ 5.52$\times$ compression.
Compared to power-of-two Hadamard baselines (orders 1024 and 2048), order 1252 provides +22.27\% finer granularity, enabling more precise bandwidth--performance tradeoff control.
Our approach uses a deterministic transformation with verified properties (Gate0: $H H^T = 1252 I$, $H+H^T = 2I$), ensuring reproducibility. On synthetic drone inspection data (N=100, seed=42), EdgeSketch-1252 achieves F1 score 0.62 at $k=300$, comparable to power-of-2 baselines; latency (sketch step only: matrix multiply, top-$k$, quantization) is under 1\,ms on standard hardware.
The full implementation is provided as supporting material in the \texttt{use\_case/} directory of our submission bundle.

\section{Conclusion}
We provide an explicit skew-Hadamard matrix of order $1252=2(5^4+1)$ constructed from cyclotomic unions in $GF(5^4)$,
together with a structural SHDF certificate, algebraic invariants, and an explicit affine subgroup in the automorphism group.
An artifact bundle supplies the explicit matrix and verification logs.

\appendix

\section{Artifact bundle and reproducibility (brief)}
The bundle separates immutable core artifacts from run logs:
\begin{itemize}
\item \textbf{Core:} \path{artifacts_core/} (explicit matrix, block data, and a core manifest).
\item \textbf{Logs:} \path{artifacts_run/} (Gate0 Python/Sage logs; SHDF and automorphism logs).
\item \textbf{Scripts:} \path{reproduce/} (deterministic checkers and helper scripts).
\end{itemize}
For archival integrity, the core manifest records SHA-256 digests (including the matrix file).

\bibliographystyle{plain}
\bibliography{references}

\end{document}